\documentclass[article, reqno]{amsart}
\usepackage{amsfonts,amssymb,amsthm,amsmath,mathrsfs,graphicx,wrapfig}
\usepackage[margin=1in]{geometry}

\newcommand{\ra}{\rightarrow}
\newcommand{\Z}{\mathbb{Z}}

\newcommand{\N}{\mathbb{N}}
\newcommand{\R}{\mathbb{R}}

\newcommand{\F}{\mathcal{F}}
\newcommand{\G}{\mathcal{G}}
\newcommand{\h}{\mathcal{H}}

\newcommand{\s}{\mathcal{S}}

\newcommand{\bv}{\left[\begin{array}{c}}
\newcommand{\ev}{\end{array}\right]}

\theoremstyle{plain}
\newtheorem{theorem}{Theorem}
\newtheorem{lemma}[theorem]{Lemma}

\newtheorem{proposition}[theorem]{Proposition}
\newtheorem{corollary}[theorem]{Corollary}

\theoremstyle{remark} 

\newtheorem{remark}[theorem]{Remark}

\newtheorem{example}[theorem]{Example}

\theoremstyle{definition} 

\newtheorem{defn}[theorem]{Definition}

\title{A Variation on H\"older-Brascamp-Lieb Inequalities}
\author{Kevin O'Neill}

\begin{document}
\large
\begin{abstract}
The H\"older-Brascamp-Lieb inequalities are a collection of multilinear inequalities generalizing a convolution inequality of Young and the Loomis-Whitney inequalities. The full range of exponents was classified in Bennett et al. \cite{MR2661170}. In a setting similar to that of Ivanisvili and Volberg \cite{MR3431655}, we introduce a notion of size for these inequalities which generalizes $L^p$ norms. Under this new setup, we then determine necessary and sufficient conditions for a generalized H\"older-Brascamp-Lieb type inequality to hold and establish sufficient conditions for extremizers to exist when the underlying linear maps match those of the convolution inequality of Young.
\end{abstract}

\maketitle




\section{Introduction}

In a dual form, Young's convolution inequality on $\R^d$ states that

\begin{equation}\label{eq:youngs}
\int_{\R^d}\int_{\R^d}f(y)g(x-y)h(x)dxdy\leq C_{p,q,r,d}||f||_p||g||_q||h||_r,
\end{equation}

where $p,q,r\in[1,\infty]$, $\frac{1}{p}+\frac{1}{q}+\frac{1}{r}=2$ (interpreting $1/\infty$ as 0) and $C_{p,q,r,d}$ is the optimal constant. 

It was established in ~\cite{MR0385456}, \cite{Lieb90}, and ~\cite{MR0412366} that certain compatible triplets of Gaussians are the extremizers of \eqref{eq:youngs}, providing a sharp form of the inequality. Later ~\cite{MR2077162}  proved this by running the heat equation through time with $f, g$, and $h$ as initial data and showing that the left hand side is nondecreasing with time.

~\cite{MR2661170} provides the following generalization of Young's inequality which also encompasses H\"older's inequality and the Loomis-Whitney inequality. Let $d, n, d_j$ be positive intergers ($1\leq j\leq n$) and let $L_j:\R^d\ra\R^{d_j}$ be surjective linear maps. Then there exists $C<\infty$ such that

\begin{equation} \label{eq:HBL}
\int_{\R^d}\prod_{j=1}^nf_j(L_j(x))dx\leq C\prod_{j=1}^n||f_j||_{L^{p_j}(\R^{d_j})},
\end{equation}

\noindent for all $f_j\in L^{p_j}(\R^{d_j})$ and with $C$ depending only on $d,n,d_j$, and $L_j$, if and only if both

\begin{equation}\label{1stcond}
\sum_{j=1}^n\frac{d_j}{p_j}=d
\end{equation}

\noindent and

\begin{equation}\label{2ndcond}
\dim(V)\leq\sum_{j=1}^n\frac{\dim(L_jV)}{p_j}
\end{equation}

\noindent for all subspaces $V\subset\R^d$. The set of exponents $(1/p_1,...,1/p_n)$ satisfying both \eqref{1stcond} and \eqref{2ndcond} is called the {\it H\"older-Brascamp-Lieb (HBL) polytope}. Thus, the HBL polytope is compact and convex with finitely many extreme points.

One may obtain ~\eqref{eq:youngs} from \eqref{eq:HBL} by setting $d=2k, n=3, d_j=k$, and $L_1(x,y)=y, L_2(x,y)=x-y, L_3(x,y)=x$, where $\R^{2k}=\{(x,y):x,y\in\R^k\}$. \cite{BCCT} proved the existence of extremizers (in particular, certain tuples of Gaussians) by a generalization of the above heat equation method.
\eqref{eq:HBL} may be rewritten in the form

\begin{equation} \label{eq:HBL2}
\int_{\R^d}\prod_{j=1}^nf_j(L_j(x))^{s_j}dx\leq C\prod_{j=1}^n\left(\int_{\R^{d_j}} f_j\right)^{s_j},
\end{equation}

\noindent where $s_j=1/p_j$ and $f_j\geq0$. (This is a nonrestricting assumption since $|\int f|\leq\int|f|$.) In this paper, we will frequently use the notation $s=(s_1,...,s_n)$. The above may be rewritten as

\begin{equation} \label{eq:HBL3}
\int_{\R^d}B(f_1(L_1(x)),...,f_n(L_n(x)))dx\leq CB\left(\int\!f_1,...,\int\!f_n\right),
\end{equation}

\noindent where $B(y_1,...,y_n)=y_1^{s_1}\cdots y_n^{s_n}$. In this paper, we will say $B:\R^n_+\ra\R_+$ is a {\it H\"older-Brascamp-Lieb (HBL) function for $\{L_j\}$} if \eqref{eq:HBL3} holds for all nonnegative $f_j\in L^1(\R^{d_j})$ . Here $\R_+=[0,\infty)$.

A similar question was explored in ~\cite{MR3431655} in the case where the $L_j$ maps are rank 1 ($d_j\equiv1$). The authors found sufficient conditions on $B$ for the left hand side of \eqref{eq:HBL3} to be bounded by the same expression where the $f_j$ are replaced with certain Gaussians $G_j$ with $\int f_j=\int G_j$. A corollary of this result is that certain tuples of Gaussians are among the extremizers. The key condition was a concavity requirement on $B$ which allowed the heat equation method from ~\cite{MR2077162} to work. Their bounding term matches our in the case where each of the $L_j$ is the identity.

In this paper, we will remove the rank 1 restriction and provide necessary and sufficient conditions for a function $B:\R_+^n\ra\R_+$ to be an HBL function in the following theorem to be proven in Section 2. Part of the proof will involve the construction of a parallelipiped with certain dimensions through a dual linear programming problem as in ~\cite{2016arXiv161105944D}.

By $A\lesssim B$, we mean that there exists a $0<C<\infty$ such that $A\leq CB$ and by $A\gtrsim B$, we mean there exists a $0<C'<\infty$ such that $A\geq CB$. $A\approx B$ means $A\lesssim B$ and $A\gtrsim B$.

\begin{theorem}\label{maintheorem}
Let $B:[0,\infty)^n\ra[0,\infty)$ be nondecreasing in each coordinateand satisfy $B(y_1,...,y_n)=0$ whenever any of the $y_j$ are 0. Let $d, d_j, 1\leq j\leq n$ be positive integers and $L_j:\R^d\ra\R^{d_j}$ surjective linear maps whose H\"older-Brascamp-Lieb polytope $\mathcal{P}$ is nonempty. Then the following are equivalent:\\

1) $B$ is an HBL function for $\{L_j\}$.\\

2) For all $0<\lambda_j,y_j<\infty$,

\begin{equation}\label{eq:bigthm2}
B(\lambda_1y_1,...,\lambda_ny_n)\lesssim \max_{s\in\mathcal{P}}\lambda^{s_1}\cdots\lambda_n^{s_n}B(y_1,...,y_n).
\end{equation}

3) For all $0<\lambda_j,y_j<\infty$,

\begin{equation}\label{eq:bigthm3}
B(\lambda_1y_1,...,\lambda_ny_n)\gtrsim \min_{s\in\mathcal{P}}\lambda^{s_1}\cdots\lambda_n^{s_n}B(y_1,...,y_n).
\end{equation}
\end{theorem}

Allowing for a change of underlying constant, each of the possible conclusions in the above theorem is invariant under multiplication of $B$ by a bounded function with bounded inverse. Thus, the theorem still holds if we replace the hypothesis that $B$ is nondecreasing in each coordinate with the weaker hypothesis that $B$ is bounded above and below by a positive multiple of a function which is nondecreasing in each coordinate.

The remainder of the paper is dedicated to the question of extremizers, and we will transfer some previous results into this newer setup. In particular, we will focus on the choice of $d, n, d_j, L_j$ used in Young's inequality to emphasize the differences in setting rather than prove statements in their most general form.

In Section 3, we will state and prove a rearrangement inequality that allows one to replace $f_j$ with their symmetric decreasing rearrangements. The proof of this uses the classical technique found in ~\cite{MR839110}, where it was shown that $\int F(f(x),g(x))dx\leq\int F(f^*(x),g^*(x))dx$ for certain $F$ satisfying a second-order condition.

In Section 4, we will show that for certain $B$, near-extremizers triples of \eqref{eq:HBL3} must be localized in scale and that these scales must be close for each function in the triple. This result is similar to the one found in ~\cite{2011arXiv1112.4875C} for the setting of $L^p$ norms and will be used in establsihing precompactness. Section 5 will piece together these arguments to establish the existence of extremizers in certain cases of HBL functions, as stated in the following theorem.

For notation, let $\vec{y}=(y_1,...,y_n)$ denote a vector in $\R^n_+$ and let $\Delta_3(B;a,b,c,d,e,f)$ denote the third order difference:
\begin{equation}
B(b,d,f)-B(a,d,f)-B(b,c,f)-B(b,d,e)+B(b,c,e)+B(a,d,e)+B(a,c,f)-B(a,c,e).
\end{equation}

\begin{theorem}\label{maintheorem2}
Let $P_i(a,b,c)=a^{1/p_i}b^{1/q_i}c^{1/r_i}$, where $p_i,q_i,r_i\in(1,\infty)$ and $1/p_i+1/q_i+1/r_i=2$. Let $B=\rho(P_1,...,P_n)$ where
\begin{equation*}
\rho(\lambda_1y_1,...,\lambda_ny_n)\leq C\max_i\lambda_i\rho(y_1,...,y_n)
\end{equation*}
 for all $0<\lambda_i,y_i<\infty$ and 
\begin{equation*}
\rho(\vec{y_1})+\rho(\vec{y_2})\leq\rho(\vec{y_1}+\vec{y_2})
\end{equation*}
for all $\vec{y_i}\in\R^n_+$. Furthermore, suppose $B$ is continuous with 

\begin{equation*}
B(0,0,0)=B(x,0,0)=B(0,y,0)=B(0,0,z)=0,
\end{equation*}

\noindent along with

\begin{equation*}
\Delta_3(B;a,b,c,d,e,f)\geq0
\end{equation*}
for all $a\leq b, c\leq d, e\leq f$.

Let $\alpha, \beta, \gamma>0$. Then, there exist $f,g,h$ which maximize

\[\iint B(f(y),g(x-y),h(x))dxdy\]

\noindent under the constraint $\int\!f=\alpha, \int\!g=\beta, \int\!h=\gamma$.
\end{theorem}

The setup of Theorem \ref{maintheorem2} includes the hypotheses of the rearrangement inequality from Section 3 as well as conditions which allow us to use some tools from the $L^p$ norms setting while also extending the conclusion to other HBL functions.

Lastly, Section 6 will provide an example of an HBL function which leads to non-Gaussian extremizers. We will prove this to be the case by showing that no Gaussian is a critical point with regards to the Euler-Lagrange equations and referencing the existence of extremizers result from Section 5.

The author would like to thank his advisor, Michael Christ, for all his support during this project.\\


\section{Necessary and Sufficient Conditions for HBL functions}

The proofs of $\eqref{eq:bigthm3}\Rightarrow(\ref{eq:bigthm2})\Rightarrow\eqref{eq:HBL3}$ are relatively straightforward so we will address those here before moving on to the more involved remaining implication.

\begin{proof}[Proof of $\eqref{eq:bigthm3}\Rightarrow\eqref{eq:bigthm2}\Rightarrow\eqref{eq:HBL3}$]

Suppose \eqref{eq:bigthm3} holds. Simultaneously replace each $y_j$ in the given inequality with $\lambda_jy_j$ and each $\lambda_j$ with $\lambda_j^{-1}$. Then \eqref{eq:bigthm2} is obtained by dividing both sides by

\[\min_{s\in\mathcal{P}}\lambda^{-s_1}\cdots\lambda_n^{-s_n}\]

\noindent and then using the fact that the reciprical of the minimum is the maximum of the recipricals.

Now suppose \eqref{eq:bigthm2} and consider nonnegative $L^1$ functions $f_j$. If any of the $f_j$ has zero integral (hence is zero a.e.), then \eqref{eq:HBL3} holds trivially, so assume $\int f_j>0$ for all $j$. Letting $g_j(x)=\frac{f_j}{\int f_j}$, we rewrite the left hand side of the desired integral inequality to obtain

\begin{equation}\label{okaydawg}
\int_{\R^d}B(f_1\circ L_1(x),...,f_n\circ L_n(x))dx=\int_{\R^d}B\left(g_1\circ L_1(x)\cdot \int f_1,...,g_n\circ L_n(x)\cdot \int f_n\right)dx.
\end{equation}

By applying \eqref{eq:bigthm2}, we may bound \eqref{okaydawg} by a constant times

\[\int_{\R^d}\max_{s\in\mathcal{P}}\left(g_1\circ L_1(x)\right)^{s_1}\cdots\left(g_n\circ L_n(x)\right)^{s_n}dx\cdot B\left(\int\!f_1,...,\int\!f_n\right).\]

Let us recall the fact that $\mathcal{P}$ is a compact, convex polytope. If $s, s'\in\mathcal{P}$, then taking any point on the segment between $s$ and $s'$ corresponds to taking a weighted geometric mean of $\lambda_1^{s_1}\cdots\lambda_n^{s_n}$ and $\lambda_1^{s'_1}\cdots\lambda_n^{s'_n}$. Thus, for any $x\in\R^d$, the above maximum may be obtained at extreme points of $\mathcal{P}$. We denote the set of extreme points of $\mathcal{P}$ as $\mathcal{P'}$. Since all terms are nonnegative, we may bound the maximum by a summation over extreme points to obtain

\begin{equation}\label{okayden}
\int_{\R^d}\max_{s\in\mathcal{P}}\left(g_1\circ L_1(x)\right)^{s_1}\cdots\left(g_n\circ L_n(x)\right)^{s_n}dx\leq\int_{\R^d}\sum_{s\in\mathcal{P'}}\left(g_1\circ L_1(x)\right)^{s_1}\cdots\left(g_n\circ L_n(x)\right)^{s_n}dx.
\end{equation}

Next, we exchange the integral with the sum and bound each of the integral terms. Since each function $g_n$ has integral equal to 1, we have

\begin{equation}\label{whatsup}
\int_{\R^d}\left(g_1\circ L_1(x)\right)^{s_1}\cdots\left(g_n\circ L_n(x)\right)^{s_n}dx\leq C_s,
\end{equation}
where $C_s$ is the optimal constant such that

\[\int_{\R^d}\prod_{j=1}^nf_j(L_j(x))dx\leq C_s\prod_{j=1}^n||f_j||_{L^{p_j}(\R^{d_j})}.\]

Since $\mathcal{P}$ has only finitely many extreme points, hence combining \eqref{okaydawg}, \eqref{okayden}), and \eqref{whatsup}.

\begin{align*}
\int_{\R^k}B(f_1\circ L_1(x),...,f_n\circ L_n(x))dx&\leq \left(\sum_{s\in\mathcal{P'}}C_s\right) B\left(\int\!f_1,...,\int\!f_n\right)\\
&=C B\left(\int\!f_1,...,\int\!f_n\right).
\end{align*}

\end{proof}

The main goal of the remainder of the section will be to prove the following lemma.

\begin{lemma}\label{parallelipiped}
Let $\lambda=(\lambda_1,...,\lambda_n)$ such that $\log\lambda_j$ are nonnegative integers. Then, there exists a parellipiped $S$ such that

\[|S|\approx\min_{(s_1,...,s_n)\in\mathcal{P}}\lambda^{s_1}\cdots\lambda_n^{s_n}\]

\noindent and

\[|L_j(S)|\leq \lambda_j,\]

\noindent where the proportionality constants are independent of $\lambda$.
\end{lemma}

To see the usefulness of Lemma \ref{parallelipiped}, let us demonstrate how it may be used to complete the proof of Theorem \ref{maintheorem}. The reduction to $\log\lambda_j$ will be established in Lemma \ref{lemma:log}.

\begin{proof}[Proof of $\eqref{eq:HBL3}\Rightarrow\eqref{eq:bigthm3}$]
Given $\lambda_j$ such that $\log\lambda_j$ are nonnegative integers, let $S$ be as in Lemma \ref{parallelipiped}. Define $f_j=y_j1_{L_j(S)}$. By plugging these $f_j$ into \eqref{eq:HBL3}, we obtain a left hand side equal to

\[|\cap_j L_j^{-1}(L_j(S))|B(y_1,...,y_n)\geq|S|B(y_1,...,y_n)=\min_{(s_1,...,s_n)\in\mathcal{P}}\lambda^{s_1}\cdots\lambda_n^{s_n}B(y_1,...,y_n)\]

\noindent and a right hand side equal to

\[B(|L_1(S)|y_1,...,|L_n(S)|y_n)\leq B(\lambda_1y_1,...,\lambda_ny_n).\]

Combining the two inequalities gives \eqref{eq:bigthm3}.\\
\end{proof}

Now we begin the proof of Lemma \ref{parallelipiped}. By taking logs of the minimum seen in \eqref{eq:bigthm3}, we reduce computing this term to a linear programming problem. Fixing $\lambda=(\lambda_1,...,\lambda_n)\in\R^n_+$, we now define the {\it primal LPP} as
\[\text{minimize }\log\lambda\cdot s=\sum_js_j\log\lambda_j \text{ over } s\in\R^n_+\]
\noindent subject to
\[\sum_js_j\cdot \dim(L_j(V))\geq \dim(V)\text{ for all }V\in{\bf E}, \hspace{.5in}d=\sum_js_jd_j, \hspace{.5in}s_j\geq0.\]

In the above, $\bf{E}$ is a finite list of subspaces which are sufficient to determine the HBL polytope. By this, we mean that \eqref{2ndcond} for only subspaces in $\bf{E}$ together with \eqref{1stcond} is sufficient to describe $\mathcal{P}$. Because of this fact, we may add a finite number of subspaces to $\bf{E}$ without changing the optimum value of $\log\lambda\cdot s$.

One may note that while we have included the restriction $s_j\geq0$, we have neglected to explicitly include the restriction $s_j\leq1$. However, this may be obtained from the existing inequalities and proper choice of subspace as follows. Subtract the restriction $\dim V\leq \sum_js_j\dim L_j(V)$ from $d=\sum_js_jd_j$ to obtain

\[(d-\dim V)\geq \sum_js_j(d_j-\dim L_j(V))\]

\noindent for all subspaces $V\subset\R^d$. Fix $1\leq j_0\leq n$ and pick $V=Ker(L_{j_0})$. By the Rank-Nullity theorem, the coefficient on $s_{j_0}$ in the above is equal to $d-\dim V$. Since all other $s_j$ are already taken to be nonnegative, $s_{j_0}\leq 1$. By taking $\bf{E}$ to include all subspaces of the form $Ker(L_j)$, we may recover the bounds $s_j\leq 1$.

Next, we prove three technical lemmas to aid us in the analysis of this linear programming problem. The first is preliminary, the second allows us to deal with only nonnegative solutions and coefficients, and the third will aid us in showing that a certain algorithm terminates.

\begin{lemma}\label{lemma:basicscaling}
If $B:\R_+^n\ra\R_+$ is an HBL function, then

\[B(R^{d_1}y_1,...,R^{d_n}y_n)\approx R^dB(y_1,...,y_n)\]

\noindent for all $0<R,y_j<\infty$.
\end{lemma}

\begin{proof}
Let $0<R,y_j<\infty$ be arbitrary. Plug in the functions $f_j=y_j1_{B_R(\R^{d_j})}$ to \eqref{eq:HBL3}. The right hand side becomes $B(R^{d_1}y_1,...,R^{d_n}y_n)$ while the left hand side scales like $R^d$, giving us the inequality

\[R^dB(y_1,...,y_n)\lesssim B(R^{d_1}y_1,...,R^{d_n}y_n).\]

Since the above holds for all $0<R,y_j<\infty$, we may simultaneously repace $R$ with $1/R$ and $y_j$ with $R^{d_j}y_j$ to obtain the reverse inequality.
\end{proof}

\begin{lemma}\label{lemma:feasibiity}
It suffices to establish \eqref{eq:bigthm3} for $\lambda_j\geq1$. That is, if $B:\R_+^n\ra\R_+$ is an HBL function and \eqref{eq:bigthm3} holds for $\lambda_j\geq1$ and $0<y_j<\infty$, then it also holds for $0<\lambda_j,y_j<\infty$.
\end{lemma}

\begin{proof}
Let $\lambda_j>1$ and $0<y_j<\infty$ be given. Choose $R>0$ sufficiently large such that $R^{d_j}\lambda_j>1$ for all $j$. Then, by Lemma \ref{lemma:basicscaling} and the fact that $d=\sum_js_jd_j$ for any $s\in\mathcal{P}$,

\begin{align*}
R^dB(\lambda_1y_1,...,\lambda_ny_n)&\approx B(R^{d_1}\lambda_1y_1,...,R^{d_n}\lambda_ny_n)\\
&\gtrsim \min_{s\in\mathcal{P}}(R^{d_1}\lambda^{s_1})\cdots(R^{d_n}\lambda_n^{s_n})B(y_1,...,y_n)\\
&=R^d\min_{s\in\mathcal{P}}\lambda_1^{s_1}\cdots\lambda_n^{s_n}B(y_1,...,y_n).
\end{align*}

Dividing both sides by $R^d$ gives the desired result.
\end{proof}

\begin{lemma}\label{lemma:log}
It suffices to establish \eqref{eq:bigthm3} for $\log\lambda_j\in\N$ for all $j$.
\end{lemma}

\begin{proof}
Choose nonnegative integers $m_j$ such that $e^{m_j}\leq\lambda_j<e^{m_j+d_j}$. (We may take the $m_j\geq0$ by the previous lemma.) Since $B$ is nondecreasing in each coordinate, we have

\[B(e^{m_1}y_1,...,e^{m_n}y_n)\leq B(\lambda_1y_1,...,\lambda_ny_n)\leq B(e^{m_1+d_1}y_1,...,e^{m_n+d_n}y_n).\]

By Lemma \ref{lemma:basicscaling}, these are uniformly comparable up to a constant multiple of $e^d$. Similarly, for any $s\in\mathcal{P}$ (in particular the minimum),

\[\Pi_j(e^{m_j})^{s_j}\leq\Pi_j\lambda_j^{s_j}<\Pi_j(e^{m_j+d_j})^{s_j}.\]

Again, these are all equivalent up to a constant multiple of $e^d$ by the relation $d=\sum_js_jd_j$ for all $s\in\mathcal{P}$. By hypothesis, we have

\[B(e^{m_1}y_1,...,e^{m_n}y_n)\gtrsim \min_{s\in\mathcal{P}}\Pi_j(e^{m_j})^{s_j}B(y_1,...,y_n).\]

By replacing the above terms with the corresponding ones involving $\lambda_j$ and adjusting the constant of proportionality, \eqref{eq:bigthm3} for $\log\lambda_j\in\N$ extends to all $\lambda_j>1$, and therefore all $\lambda_j$.
\end{proof}

Let $\dim({\bf E})=(\dim V)_{V\in{\bf E}}$. We define the {\it dual LPP} as
\[\text{maximize }y\cdot\dim(\bf{E})\]
\noindent subject to
\[y\cdot\dim(L_j({\bf E}))\leq\log\lambda_j \text{ for all }j, \hspace{.5in}y_V\geq0\text{ for all }V\neq\R^d, \hspace{.5in} y_{\R^d}\text{ free}.\]

The dual LPP relates to the primal LPP via the followinfg basic theorem from linear programming. For a source, see an introductory textbook on linear programming, such as \cite{LP}.

\begin{theorem}[Duality Theorem (special case)]
Let $A$ be an $m\times n$ matrix, $c,x\in\R^n$, and $b,y\in\R^m$ for $m,n\geq1$. Suppose that $A,b,c$ have all nonnegative entries and $\{x:Ax\leq b,x\geq0\}$ is nonempty and bounded. Then, the maximum value of $c^Tx$ subject to the constraints $Ax\leq b, x\geq0$ is equal to the minimum value of $y^Tb$ subject to the constraints $y^TA\geq c^T, y\geq0$. Furthermore, there exist optimal vectors $x,y$ for both problems.
\end{theorem}

By the above theorem, the optimal value of the dual LPP is equal to the optimal value of the primal LPP. In the remainder of this section, we will work with dual vectors $y$ to construct a parellelipiped $S$ whose volume is $e^{y\cdot\dim(\bf{E})}$. By taking the optimal value of $y\cdot\dim(\bf{E})$, we will show the volume of $S$ is $\min_{s\in\mathcal{P}}\lambda^{s_1}\cdots\lambda_n^{s_n}$. We may then translate $S$ into functions $f_j$ which we plug into \eqref{eq:HBL3} to obtain \eqref{eq:bigthm3}.

Since the remainder of this section will only involve the dual LPP with minimal reference to the primal LPP, we now make the following convention. Each dual vector $y$ is of the form $(y_V)_{V\in{\bf V}}$, where ${\bf V}$ is the set of all subspaces of $\R^d$. If ${\bf W}$ is a collection of subspaces of $\R^d$, then we say a dual vector $y$ is {\it supported on {\bf W}} if $y_V=0$ for all $V\notin{\bf W}$. Each vector $y$ that we consider will be supported on a finite list of subspaces; hence the expression $y\cdot{\bf V}$ will always be well-defined.

To begin, we will show that $y$ may be taken to be supported on a {\it flag}, which we define to be a sequence of properly nested subspaces $W_1\subsetneq W_2\subsetneq...\subsetneq W_t=\R^d$.

\begin{proposition}\label{theprop}
Let $y$ be an optimal dual vector of the dual LPP which is supported on ${\bf E}$. Then, there exists a dual vector $y'$ supported on a flag such that $y\cdot\dim{\bf E}=y'\cdot\dim\bf{V}$ {\it and} $y'\cdot\dim(L_j({\bf V}))\leq y\cdot\dim(L_j({\bf E}))\leq\log\lambda_j$. Furthermore, there exists a finite list of subpaces ${\bf E'}$ independent of $y$ such that $y'$ may be chosen to be supported on ${\bf E'}$ for any optimal dual vector $y$.
\end{proposition}

Before proving the lemma, we remark that the finiteness of ${\bf E'}$ is advantageous for the following reason. When we construct the parallelipiped $S$, we would like the volumes of $S$ and $L_j(S)$ to be porportional to the $\lambda_j$ in appropriate ways. However, the proportionality constants will depend on the arrangement of the subspaces. A priori, if one changes $\lambda_j$, then one also changes the optimal dual vector, which changes which flag $y'$ is supported on. But, limiting the subspaces to a finite list ensures that a single constant will work as the $\lambda_j$ vary. This is nontrivial, since the algorithm developed in ~\cite{2016arXiv161105944D} involves summing and intersecting subspaces. It is known ~\cite{MR0227053} that a finite list of subspaces will not necessarily generate a finite list under those operations. We work around this difficulty by performing these operations in a particular order and applying the following lemma.

\begin{lemma}\label{finiteness}
Suppose $V\subset\R^d$ is a subspace and $W_1\subset...\subset W_t$ is a flag. Then $\{V,W_1,...,W_t\}$ generates only a finite list of subspaces under the operations of repeated summation and intersection.
\end{lemma}

\begin{proof} (sketch of proof)
It suffices to list all such subspaces and show the list is closed under summation and intersection. We claim the complete list is $\{V\}\cup \{W_i, V+W_i, V\cap W_i\}_{i=1}^t\cup\{W_i+(V\cap W_j)\}_{i<j}$.

Beginning with $\{V\}\cup \{W_i, V+W_i, V\cap W_i\}_{i=1}^t$, we note that most summations and intersections are already on this list since many subspaces are contained within one another and when $S\subset T$, we have $S+T=T$ and $S\cap T=S$. The two cases which this does not cover are $W_i+(V\cap W_j)$ where $i<j$ and $(V+W_i)\cap W_j$ where $i<j$. Since $W_i\subset W_j$, these two are equal and the last type of subspace on our list.

It remains to show that intersections and summations involving subspaces of the $W_i+(V\cap W_j)$ are still on our list. Adding two such subspaces, we find that $$[W_{i_1}+(V\cap W_{j_1})]+[W_{i_2}+(V\cap W_{j_2})]=W_{\max(i_1,i_2)}+(V\cap W_{\max(j_1,j_2)}),$$ which is of the same form.

Similarly, intersecting two such subspaces, we find that $$[(W_{i_1}+V)\cap W_{j_1}]\cap[(W_{i_2}+V)\cap W_{j_2}]=(W_{\min(i_1,i_2)}+V)\cap W_{\min(j_1,j_2)},$$ which is also of the same form.
\end{proof}

To prove the proposition, we will use the following {\it basic algorithm (BA)}: Given a vector $y$ which is not supported on a flag, find two subspaces $V$ and $W$ such that neither is contained in the other and $y_V\geq y_W>0$. Set $y'_{V+W}=y_{V+W}+y_W, y'_{V\cap W}=y_{V\cap W}+y_W, y'_W=0, y'_V=y_V-y_W$. Repeat this process until the desired result.

It was shown in ~\cite{2016arXiv161105944D} that the BA terminates provided the initial $y$ has all nonnegative and rational coordinates. Furthermore,  at each step $y\cdot\dim(\bf{V})$ is preserved and $y\cdot\dim(L_j({\bf V}))$ does not increase.

\begin{proof}[Proof of Proposition \ref{theprop}]
Write ${\bf E}=(E_1,...,E_k,\R^d)$. Perform the BA on $y$ but only with respect to the coordinates $y_{E_1}$ and $y_{E_2}$. This creates a flag $W_{1,1}\subsetneq...\subsetneq W_{1,t_1}$ such that our modified $y$ is supported on $\{W_{1,1},...,W_{1,t_1},E_3,...,E_k,\R^d\}$.

Now, given a $y$ supported on a flag $W_{i,1}\subsetneq...\subsetneq W_{i,t_i}$ and the remaining original subspaces $\{E_{i+2},...,E_k\}$, we perform the BA on $y$ using only the subspaces $\{W_{i,1},...,W_{i,t_i},E_{i+2}\}$. This converts $y$ to a new dual vector supported on a flag $W_{i+1,1}\subsetneq...\subsetneq W_{i+1,t_{i+1}}$ together with $E_{i+3},...,E_k$.

Continue this process until the list of subspaces $E_i$ is exhausted, resulting in a dual vector supported solely on a flag. While $y_{\R^d}$ is excluded from modification, this does not prevent our final list from being a flag since every subspace is contained in $\R^d$.

Since the $\log\lambda_j$ are integers, we may take optimal $y$ with all rational coordinates. In addition, each coordinate used in the BA is nonnegative as $y_{\R^d}$ is excluded from such operations. Since this algorithm is solely the concatenation of the BA performed on particular collections of subspaces and the BA is known to terminate in such an instance, our algorithm terminates.

It remains to prove the claim that a finite number of subspaces are considered. Certainly in the case of a particular given $y$ this is true as only finitely many subspaces are introduced in each of a finite number of steps. However, at each inductive step there are only finitely many subspaces which can be generated from the previous subspaces by Lemma \ref{finiteness}. The total number of inductive steps is bounded by $k-1$, so the total number of subspaces may be counted via a finite tree.

\end{proof}

Now we will begin the construction of particular functions which when plugged into \eqref{eq:HBL3} will estabish \eqref{eq:bigthm3}.

\begin{defn}
Suppose a dual vector $y$ is supported on an independent collection of subspaces $Y_1,...,Y_t$ whose direct sum is $\R^d$. Define the parellipiped 

\[S_y=\left\{x\in\R^d|x=\sum_{i=1}^t\sum_{j=1}^{j_i}a_i^jv_i^j, 0\leq a_i^j\leq e^{y_{Y_i}}\right\},\]

\noindent where $\{v_i^1,...,v_i^{j_i}\}$ is a (fixed) basis for $Y_i$.
\end{defn}

We cite the following two results from ~\cite{2016arXiv161105944D}. While they were proven in the context of H\"older-Brascamp-Lieb inequalities over the integers, the proofs for the results as stated here may be obtained by simply repeating the proofs from ~\cite{2016arXiv161105944D}, but replacing $\Z$ with $\R$ and $\Z^d$ with $\R^d$. Similarly, the dependence on the subspaces $Y_i$ may be deduced by simply following the proofs.

\begin{proposition}\label{prop:volume}
Let $y$ be a dual vector supported on linearly independent subspaces $Y_1,...,Y_t$ whose direct sum is $\R^d$. Then,

\[|S_y|\approx e^{y\cdot \dim({\bf V})},\]

\noindent where the proportionality constant depends only on the $Y_i$.
\end{proposition}

\begin{lemma}\label{lemma:volume}
Let $y$ be a dual vector supported on linearly independent subspaces $Y_1,...,Y_t$ whose direct sum is $\R^d$. Let $W_i=Y_1+...+Y_i$.

Let $L:\R^d\ra\R^{d'}$ be any linear map and set $c_i=\dim(L(W_i))-\dim(L(W_{i-1}))$. Then

\[|L(S_y)|\lesssim e^{\sum y_{Y_i}c_i},\]

\noindent where the proportionality constant depends only on $L$ and the $Y_i$ (or equivalently, the $W_i$).
\end{lemma}

Now fix $y$ as the dual vector supported on a flag $W_1\subsetneq...\subsetneq W_t$ as obtained from Proposition \ref{theprop}. Choose linearly independent subspaces $Y_i$ of $W_i$ such that $Y_1+...+Y_i=W_i$ and define the dual vector $y'$ supported on $\{Y_1,...,Y_t\}$ by

\begin{equation}\label{eq:yprime}
y'_{Y_i}=y_{W_i}+...+y_{W_t}.
\end{equation}

\begin{proof}[Proof of Lemma \ref{parallelipiped}]
Fix a list of subspaces $\bf{E}$ which are sufficient to determine the HBL polytope and include $Ker(L_j)$ and all the subspaces generated in Proposition \ref{theprop}. 

Let $y$ be an optimal dual vector from the dual LPP, modified by Proposition \ref{theprop} to be supported on a flag. Define $S=S_{y'}$, where $y'$ is the dual vector obtained in \eqref{eq:yprime}. Then, by Proposition \ref{prop:volume},

\begin{align*}
|S|\approx e^{y'\cdot \dim(\bf{E})}&=e^{\sum_i(y_{W_i}+...+y_{W_t})(\dim Y_i)}\\
&=e^{\sum_iy_{W_i}(\dim Y_1+...+\dim Y_i)}=e^{\sum_iy_{W_i}\dim W_i}.
\end{align*}

Since $y$ was created from an optimal dual vector, the value of $\sum_iy_{W_i}\dim W_i$ above is optimal and hence equal to the optimal value of $s\cdot\log\lambda$ from the primal LPP, giving us the desired volume estimate.

Similarly, by Lemma \ref{lemma:volume}, 

\begin{align*}
|S|\approx e^{\sum_iy'_{Y_i}c_i}&=e^{\sum_i(y_{W_i}+...+y_{W_t})c_i}\\
&=e^{\sum_iy_{W_i}(c_1+...+c_i)}=e^{\sum_iy_{W_i}dim(L_j(W_i))}\leq e^{\log\lambda_j}=\lambda_j.
\end{align*}

\noindent where the last step follows from the constraints on dual vectors. We may obtain $|L_j(S)|\leq\lambda_j$ in place of $|L_j(S)|\lesssim\lambda_j$ by a uniform scaling of $S$ with scaling parameter dependent only on the previous proportionality constants.

\end{proof}


\section{Rearrangement Inequality}

Given a function $f:\R^d\ra\R$, let $E_f(\lambda)=\{x\in\R:f(x)\geq\lambda\}$ denote its distribution function. If $E_f(\lambda)<\infty$ for all $\lambda>0$, then let $f^*$ denote its symmetric decreasing rearrangement, that is, the unique lower semicontinuous function such that $f^*$ is radially symmetric and nonincreasing with $E_{f^*}=E_f$.

Given a function $F:\R^3\ra\R$, denote its third-order difference by
\begin{align*}
\Delta_3(F;a,b,c,d,e,f)&=F(b,d,f)-F(a,d,f)-F(b,c,f)-F(b,d,e)\\
&+F(b,c,e)+F(a,d,e)+F(a,c,f)-F(a,c,e).
\end{align*}


\begin{theorem}\label{rearrange}
Let $F:\R^3\ra\R$ be continuous and satisfy

\begin{equation}\label{eq:zerocondition}
F(0,0,0)=F(x,0,0)=F(0,y,0)=F(0,0,z)=0,
\end{equation}

\noindent along with

\begin{equation}\label{eq:rectcondition}
F(R):=\Delta_3(F;a,b,c,d,e,f)\geq0
\end{equation}
for all rectangles $R=\{(x,y,z):a\leq x\leq b, c\leq y\leq d, e\leq z\leq f\}$.

Then, for any non-negative measurable functions $f, g, h$ on $\R^d$ with finite distribution functions,

\begin{equation}\label{eq:conc}
\iint F(f(s),g(t),h(s+t))dsdt\leq\iint F(f^*(s),g^*(t),h^*(s+t))dsdt.
\end{equation}

\end{theorem}

Condition \eqref{eq:zerocondition} is simply to ensure the possibility that all integrals in the following proof are finite. If $\R$ were replaced with a finite measure space, then this condition could be dropped.

\begin{proof}
For this proof, we use the notation

\[I(f,g,h):=\iint F(f(s),g(t),h(s+t))dsdt.\]

By \cite{Saks} (pp.64-68), we may extend $F(R)$ from a measure on rectangles to a Borel measure on $\R^3$, also denoted by $F$, provided that $F$ is additive.\footnote{The book of Saks proves that $F$ extends to a Borel measure in a similar way that one proves volume of rectangles extends to Lebesgue measure. It works by constructing an outer measure $F^*$ in the typical fashion, where $F^*(E)$ is the infimum of $\sum F(R_i)$ for countable collections of rectangles $R_i$ which cover $E$, and showing that $F^*$ and $F$ agree on rectangles.

Alternatively, one may prove our rearrangement lemma by first assuming that $F\in C^3(\R^3)$, so $dF=F_{xyz}dxdydz$ is well-defined. The third-order condition is used to obtain positivity of the involved integrals. Then, one may extend the result to continuous $F$ by a standard approximation argument which takes $F$ to be the uniform limit of $C^3$ functions.
} 
Here, $F$ is additive if $F(R_1\cup R_2)=F(R_1)+F(R_2)$ for any nonoverlapping rectangles $R_1$ and $R_2$. For $F(R_1\cup R_2)$ to be pre-defined, $R_1$ and $R_2$ must have an overlapping face; without loss of generality, assume this face is parallel to the $yz$-plane. Thus, $R_1=\{(x,y,z):a_0\leq x\leq a_1, c\leq y\leq d, e\leq z\leq f\}$ and $R_2=\{(x,y,z):a_1\leq x\leq a_2, c\leq y\leq d, e\leq z\leq f\}$. By definition of $F(R)$,
\begin{align*}
F(R_1)+F(R_2)&=F(a_1,d,f)-F(a_0,d,f)-F(a_1,c,f)-F(a_1,d,e)\\
&+F(a_1,c,e)+F(a_0,d,e)+F(a_0,c,f)-F(a_0,c,e)\\
&+F(a_2,d,f)-F(a_1,d,f)-F(a_2,c,f)-F(a_2,d,e)\\
&+F(a_2,c,e)+F(a_1,d,e)+F(a_1,c,f)-F(a_1,c,e)\\
&=F(a_2,d,f)-F(a_0,d,f)-F(a_2,c,f)-F(a_2,d,e)\\
&+F(a_2,c,e)+F(a_0,d,e)+F(a_0,c,f)-F(a_0,c,e)=F(R_1\cup R_2).
\end{align*}


Let

\[R_{xyz}=\{(\alpha,\beta,\gamma):0\leq \alpha\leq x,0\leq\beta\leq y,0\leq\gamma\leq z\}\]

\noindent be a rectangle with characteristic function

\[\chi_{xyz}(\alpha,\beta,\gamma)=\Phi_{\alpha\beta\gamma}(x,y,z).\]

Then, by \eqref{eq:zerocondition}, we have

\begin{align*}
\int\chi_{xyz}(\alpha,\beta,\gamma)dF(\alpha,\beta,\gamma)&=F(R_{xyz})\\
&=F(x,y,z)-F(x,y,0)-F(x,0,z)-F(0,y,z).
\end{align*}

Now we substitute $x=f(s), y=g(t), h(s+t)$ and integrate both sides of the above to obtain

\begin{align*}
I(f,g,h)&=\iint\left[\int\Phi_{\alpha\beta\gamma}(f(s),g(t),h(s+t))dF(\alpha,\beta,\gamma)\right]dsdt\\
&=\iint F(f(s),g(t),0)dsdt+\iint F(f(s),0,h(s+t))dsdt+\iint F(0,g(s),h(s+t))dsdt
\end{align*}

The $\iint F(f(s),g(t),0)dsdt$ term is invariant under symmetrization of $f$ and $g$ since they appear as functions of independent variables. The two following terms may be dealt with similarly after a change of variables, leaving us to show the desired inequality only for the term on the first line. By Fubini's theorem,

\[\iint\left[\int\Phi_{\alpha\beta\gamma}(f(s),g(t),h(s+t))dF(\alpha,\beta,\gamma)\right]dsdt=\int J(f,g,h) dF(\alpha,\beta,\gamma).\]

\noindent where

\[J(f,g,h):=\iint\Phi_{\alpha\beta\gamma}(f(s),g(t),h(s+t))dsdt.\]

Therefore, using that $F$ is a nonnegative measure, it suffices to show

\begin{equation}\label{eq:jinequality}
J(f,g,h)\leq J(f^*,g^*,h^*).
\end{equation}

By the steps above, we have in fact shown \eqref{eq:jinequality} to be equivalent to \eqref{eq:conc}. However, note that \eqref{eq:jinequality} is a statement independent of our choice of $F$. In the case that $F(x,y,z)=xyz$, then \ref{eq:conc} is the classical Riesz rearrangement inequality, which is something we already know to be true. Hence by a series of equivalences, we have proven our theorem for any $F$.

\end{proof}

We conclude this section with the following remark. One may show by example that the third-order condition which is found as a hypothesis in the rearrangement inequality is necessary. To see this, suppose that there exist $a_1\leq a_2, b_1\leq b_2, c_1\leq c_2$ such that $F(R)<0$, where $R=\{(x,y,z):a_1\leq x\leq a_2, b_1\leq y\leq b_2, c_1\leq z\leq c_2\}$.

Let $\chi_{[s,t]}$ denote the indicator function of the interval $[s,t]$ and let $f=a_1\chi_{[-5/2,5/2]}+(a_2-a_1)\chi_{[1/2,3/2]}$, $g=b_1\chi_{[-5/2,5/2]}+(b_2-b_1)\chi_{[1/2,3/2]}$, and $h=c_1\chi_{[-5,5]}+(c_2-c_1)\chi_{[-1,1]}$. Denoting $LHS=\iint F(f(s),g(t),h(s+t))dsdt$ and $RHS=\iint F(f^*(s),g^*(t),h^*(s+t))dsdt$, then one may compute
\begin{multline*}
LHS=F(a_2,b_2,c_1)+2[F(a_2,b_1,c_2)+F(a_2,b_1,c_1)+F(a_1,b_2,c_2)+F(a_1,b_2,c_1)]+5F(a_1,b_1,c_2)\hspace{1in}\\
+11F(a_1,b_1,c_1)+F(a_1,0,c_2)+F(0,b_1,c_2)+5[F(a_2,0,c_1)+F(0,b_2,c_1)]+19[F(a_1,0,c_1)+F(0,b_1,c_1)]
\end{multline*}
and
\begin{multline*}
RHS=F(a_2,b_2,c_2)+F(a_2,b_1,c_2)+3F(a_2,b_1,c_1)+F(a_1,b_2,c_2)+3F(a_1,b_2,c_1)+6F(a_1,b_1,c_2)\hspace{1in}\\
+10F(a_1,b_1,c_1)+F(a_1,0,c_2)+F(0,b_1,c_2)+5[F(a_2,0,c_1)+F(0,b_2,c_1)]+19[F(a_1,0,c_1)+F(0,b_1,c_1)]
\end{multline*}

Thus, $RHS-LHS=F(R)<0$.\\


\section{The Scales Argument}


Let $f,g,h:\R^d\ra\R$ and write $f=\sum_{j\in\Z}2^jF_j$, where $1_{\F_j}\leq|F_j|<2\cdot1_{\F_j}$ and the $\F_j$ are disjoint subsets of $\R^d$. We may decompose $g=\sum_{k\in\Z}2^kG_k$ and $h=\sum_{l\in\Z}2^lH_l$ with associated sets $\G_k$ and $\h_l$, respectively.

For this section we introduce the following notation. If $B:\R_+^3\ra\R_+$ is measurable, then

\[I_B(f,g,h):=\iint B(f(y),g(x-y),h(x))dxdy.\]

We note that $I_B$ is a trilinear form and that \eqref{eq:HBL3} may be stated as $I_B(f,g,h)\lesssim B(\int f, \int g, \int h)$.

\begin{proposition}\label{bigprop}
Let $P_i(a,b,c)=a^{1/p_i}b^{1/q_i}c^{1/r_i}$, where $p_i,q_i,r_i\in(1,\infty)$ and $1/p_i+1/q_i+1/r_i=2$. Let $B=\rho(P_1,...,P_n)$ where
\begin{equation}\label{prop1}
\rho(\lambda_1y_1,...,\lambda_ny_n)\leq C\max_i\lambda_i\rho(y_1,...,y_n)
\end{equation}
\noindent and 
\begin{equation}\label{prop2}
\rho(\vec{y_1})+\rho(\vec{y_2})\leq\rho(\vec{y_1}+\vec{y_2}).
\end{equation}

Then there exist postive constants $\delta_0, c_0, C_0$ and positive functions $\theta, \Theta$ such that

\[\lim_{t\ra\infty}\theta(t)=0\hspace{.5in}\lim_{\delta\ra0}\Theta(\delta)=0\]

\noindent with the following properties. Let $0<\delta\leq\delta_0$ Let $f, g, h:\R^d\ra[0,\infty)$ be integrable functions with $\int\!f=\alpha, \int\!g=\beta, \int\!h=\gamma$ and

\[I_B(f,g,h)\geq(1-\delta)AB(\alpha, \beta, \gamma),\]

\noindent where $A$ is the optimal constant in the reverse inequality. Then there exist $k,k',k''\in\Z$ such that

\[2^k|\F_k|\geq c_0\]

\[\sum_{|j-k|\geq m}2^j|\F_j|\leq\theta(m)+\Theta(\delta)\]

\noindent with the analogous properties for $g$ (with $k'$ in place of $k$) and $h$ (with $k''$ in place of $k$). Lastly, we have

\[|k-k'|+|k-k''|\leq C_0.\]
\end{proposition}

\begin{remark}
It is implicit in the statement of this theorem that $B$ is an HBL function. This may be established by using \eqref{prop1} to prove \eqref{eq:bigthm2}. We also note that \eqref{prop1} is precisely the condition on $\rho$ which lets \eqref{eq:HBL3} hold with $B=\rho$ in the case that each of the $L_j$ is the identity map.
\end{remark}

\begin{proof}
Let $\eta>0$ be a small parameter and define $S=\{j\in\Z:2^j|\F_j|>\eta\}$. Let $\overline{f}=\sum_{j\in S}2^jF_j$. Note that $|S|\leq C\eta^{-1}$ by Chebyshev's inequality.

Fix $1\leq i\leq n$ and write $p=p_i, q=q_i, r=r_i$. Choose $\tilde{p}>p, \tilde{q}>q, \tilde{r}>r$ with $\frac{1}{\tilde{p}}+\frac{1}{\tilde{q}}+\frac{1}{\tilde{r}}=1$. Then, taking advantage of the disjointness of the $\F_j$, we have

\begin{align*}
||f^{1/p}-\overline{f}^{1/p}||_{L^{p,\tilde{p}}}^{\tilde{p}}&=||\sum_{j\notin S}2^{j/p}F_j^{1/p}||_{L^{p,\tilde{p}}}^{\tilde{p}}\\
&\asymp\sum_{j\notin S}(2^{j/p}|\F_j|^{1/p})^{\tilde{p}}\\
&\leq\max_{j\notin S}(2^{j/p}|\F_j|^{1/p})^{\tilde{p}-p}\sum_{j\notin S}(2^{j/p}|\F_j|^{1/p})^p\\
&\leq\eta^{\frac{\tilde{p}-p}{p}}\sum_{j\notin S}(2^{j/p}|\F_j|^{1/p})^p\\
&\leq C\eta^{\frac{\tilde{p}-p}{p}}||f^{1/p}-\overline{f}^{1/p}||_{L^p}^p.
\end{align*}

Now define $S(\eta)=S\times\Z\times\Z$. Taking advantage of the classical inequality

\begin{equation*}
\langle f*g,h\rangle\leq C||f||_{L^{p,\tilde{p}}}||g||_{L^q}||h||_{L^r},
\end{equation*}
we see that

\begin{align*}
I_{P_i}(f-\overline{f}, g, h)&=\sum_{S(\eta)}2^{j/p_i+k/q_i+l/r_i}\langle F_j^{1/p_i}*G_k^{1/q_i},H_l^{1/r_i}\rangle\\
&\leq C||f^{1/p}-\overline{f}^{1/p}||_{L^p}\\
&\leq C\eta^{\gamma_i},
\end{align*}
where $\gamma_j=\frac{\tilde{p_i}-p_i}{p_i\tilde{p_i}}>0$. 

By disjointness of supports of $f$ and $\overline{f}$,

\begin{equation*}
I_B(f,g,h)=I_B(\overline{f},g,h)+I_B(f-\overline{f},g,h)
\end{equation*}

By Theorem \ref{maintheorem},

\[\iint\rho(f_1(x,y),...,f_n(x,y))dxdy\leq C\rho\left(\int f_1,...,\int f_n\right).\]

Thus,
\begin{align*}
I_B(f-\overline{f},g,h)&\leq\rho\left[I_{P_1}(f-\overline{f},g,h),...,I_{P_n}(f-\overline{f},g,h)\right]\\
&\leq C\rho(C_1\eta^{\gamma_1},...,C_n\eta^{\gamma_n})\\
&\leq C\eta^{\min\gamma_i}
\end{align*}
and
\begin{equation}\label{previousbound}
I_B(f-\overline{f},g,h)\leq C\eta^\gamma
\end{equation}
for some fixed $\gamma>0$.

As $\eta\ra0$, the left hand side of \eqref{previousbound} approaches 0. However, we are given that $f$ is a near-maximizer of this integral, so $\overline{f}\neq0$ and $S\neq\emptyset$. This establishes our first conclusion.

For our next conclusions, we will find an upper bound on the diameter of $S$,

\[M=\max_{j,j'\in S}|j-j'|.\]

Let $N$ be a large positive integer. Then there exist integers $I^\flat<I^\sharp$ such that $S\cap(-\infty,I^\flat]\neq\emptyset, S\cap[I^\sharp,\infty)\neq\emptyset, S\cap(I^\flat,I^\sharp)=\emptyset, I^\sharp-I^\flat\geq M/(2N|S|)$, and, denoting $f_0=\sum_{I^\flat<j<I^\sharp}2^jF_j$,

\[\int|f_0|\leq N^{-1}\int|f-\overline{f}|\leq CN^{-1}\eta^c.\]

Additionally, we may take $I^\sharp-I^\flat$ to be divisible by 2. Now define

\[f^\sharp=\sum_{j\geq I^\sharp}2^jF_j,\hspace{.5in}f^\flat=\sum_{j\leq I^\flat}2^jF_j\]

\noindent so that $f=f^0+f^\sharp+f^\flat$. Next, let $I=(I^\sharp+I^\flat)/2$ and define

\[g^\sharp=\sum_{k\geq I}2^kG_k,\hspace{.5in}h^\sharp=\sum_{l\geq I}2^lH_l,\]

\noindent and $g^\flat=g-g^\sharp, h^\flat=h-h^\sharp$. We will shortly be analyzing the expression

\begin{equation}\label{eq:crossterms}
\langle (f-f^0)^{1/p}*g^{1/q},h^{1/r}\rangle=\langle(f^\sharp+f^\flat)^{1/p}*(g^\sharp+g^\flat)^{1/q},(h^\sharp+h^\flat)^{1/r}\rangle
\end{equation}

\noindent so let us first prove the following lemma.

\begin{lemma}\label{lemma:mixed}
There exist constants $c>0$ and $C<\infty$ such that each of the mixed terms in the expansion of \eqref{eq:crossterms} is $\leq C2^{-c\eta M/N}$.
\end{lemma}
Note that while \eqref{eq:crossterms} involves nonlinear expressions, we may take a natural multilinear expansion of it since $f^\sharp$ and $f^\flat$ have disjoint supports, hence $(f^\sharp+f^\flat)^{1/p}=(f^\sharp)^{1/p}+(f^\flat)^{1/p}$ and so on. To prove the above lemma, we will make use of the following result from ~\cite{2011arXiv1112.4875C}.

\begin{lemma}\label{previoustwolemmas}
Let $p,q,r\in(1,\infty)$ with $1/p+1/q+1/r=2$. There exists $\tau>0$ and $C<\infty$ such that

\begin{equation}\label{eq:lemma11.3}
\langle 1_\F*1_\G,1_\h\rangle\leq C\left[\min_{x,y\in\{|\F|,|\G|,|\h|\}}\frac{x}{y}\right]^\tau|\F|^{1/p}|\G|^{1/q}|\h|^{1/r}
\end{equation}
for all measurable subsets $\F,\G,\h$ of $\R$ with finite measure.
\end{lemma}

\begin{proof}[Proof of Lemma \ref{lemma:mixed}.]
Consider the mixed term $\langle (f^\sharp)^{1/p}*(g^\flat)^{1/q},(h^\sharp)^{1/r}\rangle$ and let $\mathcal{S}$ be the set of multi-indices $(j,k,l)$ such that $j\geq I^\sharp$ and $k<I$. Let $\epsilon>0$ and $\s^\dagger\subset \s$ be the set of $(j,k,l)$ such that $2^{j/p}|\F_j|^{1/p}\geq\epsilon, 2^{k/q}|\G_k|^{1/q}\geq\epsilon$, and $2^{l/r}|\h_l|^{1/r}\geq\epsilon$. Note that $|\s^\dagger|\leq C\epsilon^{-3}$, a bound which may be obtained by the same reasoning as our bound on $|S|$. By \eqref{previousbound}, we have

\begin{equation}\label{eq:snotdagger}
\sum_{S\setminus S^\dagger}2^{j/p+k/q+l/r}\langle 1_{\F_j}*1_{\G_k},1_{\h_l}\rangle\leq C\epsilon^\gamma.
\end{equation}

If $(j,k,l)\in\s^\dagger$, then $2^{j/p}|\F_j|^{1/p}\leq C$ and $2^{k/q}|\G_k|^{1/q}$. The fact that $(j,k,l)\in\s$ implies

\[j\geq I^\sharp\geq I+\frac{1}{4}M/N|S|\geq I+c\eta M/N,\]

\noindent so

\[|\F_j|\leq C2^{-j}\leq C2^{-I}2^{-(i-I)}\leq C2^{-I}2^{-c\eta M/N}.\]

Also, since $k\leq I$, we have

\[|\G_k|\geq c2^{-j}\epsilon^q.\]

Therefore,

\[\frac{|\F_j|}{|\G_k|}\leq C\epsilon^{-q}2^{-c\eta M/N}\]

\noindent and \eqref{eq:lemma11.3} implies

\begin{equation}\label{eq:sdagger}
\sum_{\s^\dagger}2^{j/p+k/q+l/r}\langle 1_{\F_j}*1_{\G_k},1_{\h_l}\rangle\leq C\epsilon^{-C}2^{-c\eta M/N}.
\end{equation}

Combining \eqref{eq:snotdagger} with \eqref{eq:sdagger} and choosing $\epsilon$ small enough gives

\[\sum_{\s}2^{j/p+k/q+l/r}\langle 1_{\F_j}*1_{\G_k},1_{\h_l}\rangle\leq C2^{-c\eta M/N}.\]

This implies the lemma for both $f^\sharp,g^\flat, h^\sharp$ and $f^\sharp, g^\flat, h^\flat$. All other mixed terms may be dealt with similarly.

\end{proof}

We now observe a simple corollary to the above lemma:

\begin{align*}
I_B(f^\sharp,g^\flat,h^\sharp)&=\iint \rho(P_1(f^\sharp(y),g^\flat(x-y),h^\sharp(x)),...,P_n(f^\sharp(y),g^\flat(x-y),h^\sharp(x)))dxdy\\
&\leq C\rho(I_{P_1}(f^\sharp,g^\flat,h^\sharp),...,I_{P_n}(f^\sharp,g^\flat,h^\sharp))\\
&\leq C\rho(C2^{-c\eta M/N},...,C2^{-c\eta M/N})\\
&\leq C2^{-c\eta M/N}
\end{align*}

This will allow us to deal with the mixed terms that show up in our particular case.

We are almost ready to complete the proof of Proposition \ref{bigprop}, but we will need to employ the use of the following lemma, which deals with the power cases inside $\rho$. It is proven in ~\cite{2011arXiv1112.4875C} in the form where $f\in L^p, g\in L^q, h\in L^r$.

\begin{lemma}\label{shortChrist}
Let $P(y_1,y_2,y_3)=y_1^{1/p}y_2^{1/q}y_3^{1/r}$, where $1<p,q,r<\infty$. Let $f^\sharp, f^\flat, g^\sharp, g^\flat, h^\sharp, h^\flat$, and $\eta$ be as before. Then, there exist constants $c,\gamma>0$, depending only on $p,q,r$ such that

\begin{equation}
P\left(\textstyle\int f^\sharp, \int\!g^\sharp, \int\!h^\sharp\right)+P\left(\textstyle\int\!f^\flat, \int\!g^\flat, \int\!h^\flat\right)\leq (1-c\eta^\gamma)P\left(\textstyle\int f, \int\!g, \int\!h\right).
\end{equation}

\end{lemma}

Now, let $A$ be the optimal constant such that $\iint B(f(y),g(x-y),h(x))dxdy\leq AB(\int f, \int g, \int h)$. We apply Lemma \ref{lemma:mixed} and the disjointness of supports for $f^\sharp,f^\flat,f_0$ to observe that

\begin{equation}\label{okay}
I_B(f,g,h)\leq AB\left(\textstyle\int f^\sharp, \int\!g^\sharp, \int\!h^\sharp\right)+AB\left(\textstyle\int f^\flat, \int\!g^\flat, \int\!h^\flat\right)+AB\left(\textstyle\int f_0,\int\!g,\int\!h\right)+C2^{-c\eta M/N}.
\end{equation}

We deal with the $f_0$ term as follows:

\begin{align*}
B\left(\textstyle\int f_0,\int\!g,\int\!h\right)&=\rho\left[P_1\left(\textstyle\int f_0,\int\!g,\int\!h),...,P_n(\textstyle\int f_0,\int\!g,\int\!h\right)\right]\\
&=\rho((CN^{-1}\eta^c)^{1/p_1}\beta^{1/q_1}\gamma^{1/r_1},...(CN^{-1}\eta^c)^{1/p_n}\beta^{1/q_n}\gamma^{1/r_n})\\
&\leq CN^{-1}\eta^c.
\end{align*}

Now we analyze the first two terms of \eqref{okay}. We begin by using the definition of $\rho$, along with \eqref{prop2} to combine everything into a single term containing just $\rho$ and terms found in Lemma \ref{shortChrist}.

\begin{multline*}
B\left(\textstyle\int f^\sharp, \int\!g^\sharp, \int\!h^\sharp\right)+B\left(\textstyle\int f^\flat, \int\!g^\flat, \int\!h^\flat\right)\hfill\\
\hfill\leq\!\rho\!\left[P_1\!\left(\textstyle\int f^\sharp,\!\int\!g^\sharp,\!\int\!h^\sharp\!\right)\!,...,P_n\!\left(\textstyle\int f^\sharp,\!\int\!g^\sharp,\!\int\!h^\sharp\!\right)\right]+\!\rho\!\left[\!P_1\!\left(\textstyle\int f^\flat,\!\int\!g^\flat,\!\int\!h^\flat\!\right)\!,...,\!P_n\!\left(\textstyle\int f^\flat,\!\int\!g^\flat,\!\int\!h^\flat\right)\right]\\
\hfill\leq \!\rho\!\left[\!P_1\!\left(\textstyle\int f^\sharp,\! \int\!g^\sharp, \!\int\!h^\sharp\!\right)\!+\!P_1\!\left(\textstyle\int f^\flat, \!\int\!g^\flat, \!\int\!h^\flat\!\right)\!,...\!,\!P_n\!\left(\textstyle\int f^\sharp, \!\int\!g^\sharp,\! \int\!h^\sharp\!\right)\!+\!P_n\!\left(\textstyle\int f^\flat,\! \int\!g^\flat, \!\int\!h^\flat\!\right)\!\right]\!.
\end{multline*}

Next, we apply Lemma \ref{shortChrist}, then use \eqref{prop1} before returning $B$ to the expression:
\begin{align*}
B\left(\textstyle\int f^\sharp, \int\!g^\sharp, \int\!h^\sharp\right)&+B\left(\textstyle\int f^\flat, \int\!g^\flat, \int\!h^\flat\right)\\
&\leq \rho\left[(1-c_1\eta^{\gamma_1})P_1\left(\textstyle\int f, \int\!g, \int\!h\right),...,(1-c_n\eta^{\gamma_n})P_n\left(\textstyle\int f, \int\!g, \int\!h\right)\right]\\
&\leq (1-c\eta^\gamma)\rho\left[P_1\left(\textstyle\int f, \int\!g, \int\!h\right),...,P_n\left(\textstyle\int f, \int\!g, \int\!h\right)\right]\\
&=(1-c\eta^\gamma)B\left(\textstyle\int f, \int\!g, \int\!h\right),
\end{align*}

\noindent where $\gamma=\min_i\gamma_i$ as before.

In summary, we now have:

\[A(1-\delta)B\left(\textstyle\int f,\int\!g,\int\!h\right)\leq I_B(f,g,h)\leq A(1-c\eta^\gamma)B\left(\textstyle\int f, \int\!g, \int\!h\right)+CN^{-1}\eta^c+C2^{-c\eta M/N},\]

\noindent the first inequality due to the fact that $(f,g,h)$ is a near-extremizing triplet. Thus,
\[2^{-c\eta M/N}\geq c\eta^\gamma-cN^{-1}\eta^c-C\delta\geq c\eta^\gamma-cN^{-1}-C\delta.\]

We now choose $N$ to be the integer closest to a sufficiently small multiple of $\eta^{-\gamma}$ so that

\[2^{-c\eta^{1+\gamma}M}\geq c\eta^\gamma-C\delta,\]

\noindent so if $C_0$ is chosen large enough we have $\eta\geq C_0\delta^{1/\gamma}$ implies $M\leq C\eta^{-1-\gamma}(\log\eta)^{-1}$. This completes the proof of the proposition for $f$ and functions $g$ and $h$ may be taken care of similarly.

\end{proof}

\begin{corollary}
Let $S$ be a compact subset of $(1,\infty)^3$ and let $\{B_k\}_{k=1}^\infty$ be a sequence of functions satisfying the hypotheses of Proposition \ref{bigprop} such that the triples of exponents found in the $P_i$ are each contained in $S$ and such that $\lim_{k\ra\infty}B_k$ exists, where the limit is taken pointwise. Then the conclusions of Proposition \ref{bigprop} hold with $B=\lim_{k\ra\infty}B_k$.
\end{corollary}

\begin{proof}
All but one of the main steps in the proof of the main proposition involves bounding an integral of $B$. This step may be repeated with Fatou's lemma as

\[\iint B(*)dxdy\leq \liminf_{k\ra\infty} \iint B_k(*)dxdy,\]

\noindent where $*$ represents any appropriate collection of functions and the arguments (either $(f,g,h)$, or $(f^\flat,g^\flat,h^\flat)$, etc.). The one remaining step is completed using the containment of power triples within a compact subset. This allows the $\gamma_k$ and $c_k$ to not approach 0 or infinity in the limit.

\begin{align*}
AB\left(\textstyle\int f^\sharp,\!\int\!g^\sharp, \!\int\!h^\sharp\right)+AB\left(\textstyle\int f^\flat,\!\int\!g^\flat, \!\int\!h^\flat\right)&=A\lim_{k\ra\infty}B_k\left(\textstyle\int f^\sharp,\!\int\!g^\sharp, \!\int\!h^\sharp\right)+B_k\left(\textstyle\int f^\flat,\!\int\!g^\flat, \!\int\!h^\flat\right)\\
&\leq A\liminf_{k\ra\infty}(1-c_k\eta^{\gamma_k})B_k\left(\textstyle\int f, \int\!g, \int\!h\right)\\
&\leq A(1-c\eta^\gamma)B\left(\textstyle\int f, \int\!g, \int\!h\right),
\end{align*}
where $c_k$ and $\gamma_k$ are the appropriate constants corresponding to $B_k$.
\end{proof}

\begin{example}
The main proposition applies to $B(y_1,y_2,y_3)=\int_{-1/6}^{1/6}y_1^{2/3-t/2}y_2^{2/3-t/2}y_3^{2/3+t}dt$.\\
\end{example}

\section{Existence of Extremizers}
%
%
%
%

%
%
%
%

Following \cite{2011arXiv1112.4875C}, we introduce the following definitions.

\begin{defn}
Let $\theta:\R^+\ra\R^+$ be continuous such that $\lim_{\rho\ra\infty}\theta(\rho)=0$. Then a function $f\in L^1(\R^d)$ is {\it normalized with norm $\alpha$ with respect to $\theta$} if $\int f=\alpha$ and 

\[\int_{|f(x)|>\rho}|f(x)|dx\leq\theta(\rho) \text{ for all } \rho<\infty\]

\[\int_{|f(x)|<\rho^{-1}}|f(x)|dx\leq\theta(\rho) \text{ for all } \rho<\infty.\]

If $\eta>0$, then $f\in L^1(\R^d)$ is {\it $\eta$-normalized with respect to $\theta$} if there exists a decomposition $f=g+b$ where $g$ is normalized with respect to $\theta$ and $||b||_1<\eta$.

\end{defn}

Under the above definitions, our main proposition from Section 4 states that any extremizing sequence $\{(f_n,g_n,h_n)\}_{n=1}^\infty$ for $\iint B(f(y),g(x-y),h(x))dxdy$ may be dilated such that all $f_n, g_n,$ and $h_n$ are $\eta$-normalized with their original norms and with respect to the same $\theta$ with $\eta\ra0$ as $n\ra\infty$. While this is trivial in the setting involving $L^p$ norms, here we must reference Lemma \ref{lemma:basicscaling}, which says $B(\lambda y_1,\lambda y_2,\lambda y_3)=\lambda^2B(y_1,y_2,y_3)$. Thus, we obtain the dilation symmetry

\[\iint B(\lambda f(\lambda y),\lambda g(\lambda(x-y)),\lambda h(\lambda x))dxdy=\iint B(f(y),g(x-y),h(y))dxdy.\]

One may now take each triple $(f_n,g_n,h_n)$ to be at the same scale by application of the dilation symmetry.

We now begin our proof of Theorem \ref{maintheorem2}.

\begin{proof}
Let $\{(f_n,g_n,h_n)\}_{n=1}^\infty$ be an extremizing sequence satisfying $\int f_n=\alpha, \int g_n=\beta, \int h_n=\gamma$ for all $n\geq1$. By Theorem \ref{rearrange} (and a suitable change of coordinate), we may replace $f_n,g_n,h_n$ with $(f_n^*,g_n^*,h_n^*)$ to obtain another extremizing sequence consisting of functions which are radially symmetric and nonincreasing.

By Proposition \ref{bigprop} and the dilation symmetry, we may replace the extremizing sequence with one which is $\eta$-normalized with respect to a continuous function $\theta:\R^+\ra\R^+$, where $\eta\ra0$ as $n\ra\infty$. (The benefit here is that we may use the same $\theta$ for all triples in our sequence.) In the sequel, $\{(f_n,g_n,h_n)\}_{n=1}^\infty$ will denote the new, normalized, symmetrized sequence. To complete the proof, it suffices to show that each of $\{f_n\}, \{g_n\}, \{h_n\}$ are precompact.

Let $\epsilon>0$. For any $\rho<\infty$ and $0<A<\infty$ we have

\[\int_{|t|\leq A}f_n(t)dt\leq c_d\rho A+\int_{f_n>\rho}f_n.\]

Since $f_n$ is $\eta$-normalized with $\eta\ra0$, there exist $\rho$ and $N$ large enough such that $n>N$ implies

\[\int_{f_n>\rho}f_n<\epsilon/2.\]

By choosing $A$ small enough, we have
\begin{equation}\label{eq:Asmall}
\int_{|t|\leq A}f_n(t)dt<\epsilon
\end{equation}

\noindent for sufficiently large $n$. Now let $0<B<\infty$. By the fact that for symmetric decreasing $f_n$ with $\int f_n=\alpha$ implies $f_n(s)\leq c_d\alpha|s|^{-d}$, we have

\[\int_{|t|\geq B}f_n(t)\leq\int_{|t|\geq B}f_n\leq c_d\alpha B^{-d}f_n(t)dt\leq \theta(c_d^{-1}\alpha^{-1}B^d)+o(1),\]

\noindent where $o(1)\ra0$ as $n\ra\infty$. Since $\theta(\rho)\ra0$ as $\rho\ra\infty$, we may take $B$ large enough that

\begin{equation}\label{eq:Blarge}
\int_{|t|\geq B}f_n(t)dt<\epsilon
\end{equation}

\noindent for sufficiently large $n$. Fixing $0<A<B<\infty$, we see that the restrictions of $f_n$ to $[A,B]$ are radial symmetric decreasing with $0\leq f_n(t)\leq c_d\alpha A^{-d}$ so they are precompact in $L^1$ on $\{t\in\R^d:A\leq|t|\leq B\}$. By \eqref{eq:Asmall} and \eqref{eq:Blarge}, $\{f_n\}$ is precompact in $L^1(\R^d)$. By the same reasoning, $\{g_n\}$ and $\{h_n\}$ are precompact in $L^1(\R^d)$ as well, which completes the proof.\\
\end{proof}


\section{Non-Gaussian Extremizers}

In the classical version of Young's inequality, it is known that extremizers exist for the entire (possible) range of exponents and furthermore, those extremizers are always Gaussians. In ~\cite{MR3431655}, it is shown that for a certain class of functions $B$, there exist maximizers of

\[\iint B(f(y),g(x-y),h(x))dxdy\]

\noindent and that these maximizers are always Gaussians. However, the follow proposition shows that our expansion of the class of functions $B$ breaks this pattern.

\begin{proposition}
Fix $\alpha, \beta, \gamma>0$. There exists a $B:\R^3\ra\R$ satisfying the hypotheses of Theorem \ref{maintheorem2} such that under the constraints $\int\!f=\alpha, \int\!g=\beta, \int\!h=\gamma$, there exist maximizers of

\[\iint B(f(y),g(x-y),h(x))dxdy\]
which are not all Gaussians.
\end{proposition}

The proof of this proposition is based on a simple use of Euler-Lagrange equations, though some aspects are modified to fit our particular setting. Extremizers exist due to results from previous sections and extremizers must also be critical points of the functional $\int B(f(x),g(x-y),h(y))$. However, any critical point must satisfy the Euler-Lagrange equations and it will be clear that no collection of Gaussians does. Before going any further, let us define a {\it critical point} as a triplet of $L^1$ functions $(f,g,h)$ such that for any $j\in C_c^\infty$ with $\int j=0$,

\[\iint B(f(y)+tj(y),g(x-y),h(x))dxdy=\iint B(f(y),g(x-y),h(x))dxdy+o(|t|)\]

\noindent as $\epsilon\ra0$ and that the analogous equation holds with perturbations of $g$ and $h$. The reason we add the restiction that $\int j=0$ is so that $\int (f+j)=\int f=\alpha$ and $f+j$ satisfies the appropriate constraint. The condition that $j$ is bounded with compact support is to ensure convergence of certain integrals which arise in the following proof.

\begin{proof}
Let $B(y_1,y_2,y_3)=y_1^{1/p_1}y_2^{1/p_2}y_3^{1/p_3}+y_1^{1/q_1}y_2^{1/q_2}y_3^{1/q_3}$, where

\[\frac{1}{p_i}+\frac{1}{q_i}+\frac{1}{r_i}=2\]

\noindent and $p_i,q_i,r_i\in(1,\infty)$ for $i=1,2$, but $(p_1,q_1,r_1)\neq(p_2,q_2,r_2)$. Suppose, to the contrary, that there exists Gaussians $f,g,h$ which are maximizers of $\iint B(f(y),g(x-y),h(x))dxdy$. Then, $f,g,h$ must also form a critical point. Taking the binomial expansion of $(f+tj)^{1/p_1}$, we find

\begin{align*}
\iint\! (f(y)\!+\!tj(y))^{1/p_1}g^{1/q_1}(x-y)h^{1/r_1}(x)dxdy&\!=\!\iint f^{1/p_1}(y)g^{1/q_1}(x-y)h^{1/r_1}(x)dxdy\\
&+\frac{t}{p_1}\iint f^{1/p_1-1}(y)j(y)g^{1/q_1}(x-y)h^{1/r_1}(x)dxdy\\
&+\!O\!\left(t^2\!\iint\! f^{1/p_1-2}(y)j^2(y)g^{1/q_1}(x-y)h^{1/r_1}(x)dxdy\!\right)\!.
\end{align*}

The left hand side is well-defined since $f$ is bounded below by a positive constant on the domain of $j$. Thus, we may take $t$ small enough that $f+tj>0$ everywhere. Furthermore, the integrals on the right hand side are convergent since $j$ is bounded with compact support and $1/f$ is bounded on the support of $j$. In fact, $f^{1/p_1-1}j\in L^p$ for all $1\leq p\leq\infty$. Thus,

\[\iint f^{1/p_1-1}(y)j(y)g^{1/q_1}(x-y)h^{1/r_1}(x)dxdy+\iint f^{1/p_2-1}(y)j(y)g^{1/q_2}(x-y)h^{1/r_2}(x)dxdy=0\]

\noindent for all bounded $j$ with compact support with $\int j=0$. This implies that

\[f^{1/p_1-1}(\tilde{g}^{1/q_1}*h^{1/r_1})+f^{1/p_2-1}(\tilde{g}^{1/q_2}*h^{1/r_2})=C\]

\noindent for some constant $C$, where $\tilde{g}(x)=g(-x)$. There are now two cases. The first is that neither of the 2 summed terms is constant, in which case each is either a Gaussian or the inverse of a Gaussian and their sum cannot be constant. The second case is that each of the two terms is consant. However, since $(p_1,q_1,r_1)\neq(p_2,q_2,r_2)$, this is impossible to obtain with the same Gaussians for each term. Thus, Gaussians cannot be critical points (or maximizers) for $\iint B(f(y),g(x-y),h(x))dxdy$ with the given constraints.
\end{proof}


\bibliographystyle{plain}
\bibliography{Bibliography1}

\begin{thebibliography}{10}

\bibitem{MR0385456}
William Beckner.
\newblock Inequalities in {F}ourier analysis.
\newblock {\em Ann. of Math. (2)}, 102(1):159--182, 1975.

\bibitem{BCCT}
Jonathan Bennett, Anthony Carbery, Michael Christ, and Terence Tao.
\newblock The brascamp--lieb inequalities: Finiteness, structure and extremals.
\newblock {\em GAFA}, 17:1343--1415, 2008.

\bibitem{MR2661170}
Jonathan Bennett, Anthony Carbery, Michael Christ, and Terence Tao.
\newblock Finite bounds for {H}\"older-{B}rascamp-{L}ieb multilinear
  inequalities.
\newblock {\em Math. Res. Lett.}, 17(4):647--666, 2010.

\bibitem{MR0227053}
Garrett Birkhoff.
\newblock {\em Lattice theory}.
\newblock Third edition. American Mathematical Society Colloquium Publications,
  Vol. XXV. American Mathematical Society, Providence, R.I., 1967.

\bibitem{MR0412366}
Herm~Jan Brascamp and Elliott~H. Lieb.
\newblock Best constants in {Y}oung's inequality, its converse, and its
  generalization to more than three functions.
\newblock {\em Advances in Math.}, 20(2):151--173, 1976.

\bibitem{MR2077162}
E.~A. Carlen, E.~H. Lieb, and M.~Loss.
\newblock A sharp analog of {Y}oung's inequality on {$S^N$} and related entropy
  inequalities.
\newblock {\em J. Geom. Anal.}, 14(3):487--520, 2004.

\bibitem{2011arXiv1112.4875C}
M.~{Christ}.
\newblock {Near-extremizers of Young's Inequality for R\^{}d}.
\newblock {\em ArXiv e-prints}, December 2011.

\bibitem{MR839110}
J.~A. Crowe, J.~A. Zweibel, and P.~C. Rosenbloom.
\newblock Rearrangements of functions.
\newblock {\em J. Funct. Anal.}, 66(3):432--438, 1986.

\bibitem{2016arXiv161105944D}
James Demmel and Alex Rusciano.
\newblock Parallelepipeds obtaining hbl lower bounds.
\newblock Technical Report UCB/EECS-2016-162, EECS Department, University of
  California, Berkeley, Nov 2016.

\bibitem{LP}
Thomas~S. Ferguson.
\newblock Linear programming https://www.math.ucla.edu/~tom/lp.pdf.

\bibitem{MR3431655}
P.~Ivanisvili and A.~Volberg.
\newblock Hessian of {B}ellman functions and uniqueness of the
  {B}rascamp-{L}ieb inequality.
\newblock {\em J. Lond. Math. Soc. (2)}, 92(3):657--674, 2015.

\bibitem{Lieb90}
E.H. Lieb.
\newblock {Gaussian Kernels have only Gaussian Maximizers}.
\newblock {\em Invent. Math.}, 102:179--208, 1990.

\bibitem{Saks}
Stanislaw Saks.
\newblock {\em Theory of the Integral}.
\newblock Second Revised Edition. Hafner Publishing Company, New York, 1937.

\end{thebibliography}

\end{document}